\documentclass{article}

\usepackage[utf8]{inputenc}
\usepackage{amsmath,amssymb,amsthm}
\usepackage{hyperref}

\DeclareMathOperator{\Aut}{Aut}
\DeclareMathOperator{\Hom}{Hom}
\DeclareMathOperator{\Id}{Id}
\DeclareMathOperator{\ev}{ev}
\DeclareMathOperator{\Contr}{Contr}
\DeclareMathOperator{\diag}{diag}

\newcommand{\Mbar}{\overline{M}}
\newcommand{\CC}{\mathbb C}
\newcommand{\QQ}{\mathbb Q}
\newcommand{\ZZ}{\mathbb Z}
\newcommand{\PP}{\mathbb P}
\newcommand{\cR}{\mathcal R}
\newcommand{\tH}{\widetilde H}
\newcommand{\strata}{\mathcal S}
\newcommand{\vir}{\mathrm{vir}}

\newtheorem{thm}{Theorem}
\newtheorem{proposition}{Proposition}
\theoremstyle{definition}
\newtheorem{defi}{Definition}
\newtheorem{ex}{Example}
\theoremstyle{remark}
\newtheorem{remark}{Remark}

\begin{document}

\title{Relations on $\Mbar_{g,n}$ via equivariant Gromov-Witten theory
  of $\PP^1$}
\author{Felix Janda}
\maketitle
\abstract{
  We give a proof of Pixton's generalized Faber--Zagier relations in
  the tautological Chow ring of $\Mbar_{g,n}$.
  The strategy is very similar to the work of
  Pandharipande--Pixton--Zvonkine, who have given a proof of the same
  result in cohomology.
  The main tool is the Givental--Teleman classification of semisimple
  cohomological field theories.
  While in general only known in cohomology, the Givental--Teleman
  classification for the equivariant Gromov--Witten theory of the
  projective line is also valid in the Chow ring, as can be shown via
  virtual localization.
  We obtain the relations just from the latter theory.
}

\section{Introduction}

The study of the Chow ring of the moduli space $M_g$ of genus $g$
algebraic curves was initiated by Mumford in \cite{Mu83}.
For this he introduced tautological classes, which reflect the
geometry of the objects parametrized by the moduli space.
The tautological ring $R^*(M_g)$ is the ring generated by tautological
classes and the definitions were later extended to the Deligne--Mumford
compactification $\Mbar_{g, n}$ of stable curves of genus $g$ with $n$
markings, where $R^*(\Mbar_{g, n})$ is compactly defined \cite{FaPa05}
as the smallest system
\begin{equation*}
  R^*(\Mbar_{g, n}) \subseteq A^*(\Mbar_{g, n})
\end{equation*}
of subrings compatible with push-forward under the tautological maps,
that is the maps obtained from forgetting marked points or gluing
curves along markings.
See the recent survey article \cite{Pa16P} for a discussion of the
tautological ring and the topics surrounding this paper.

There is an explicit set of generators of the ring $R^*(\Mbar_{g, n})$
of the form
\begin{equation}
  \label{eq:tautgen}
  \xi_{\Gamma*}\left(\prod_{v \in \Gamma} P_v\right),
\end{equation}
which we briefly describe now.
$\Gamma$ is a dual graph describing a stratum of $\Mbar_{g, n}$ by the
topological type of its generic elements; that is the vertices $v$ of
$\Gamma$ correspond to irreducible components and are labeled by a
genus $g(v)$, edges correspond to nodes and there are numbered legs
corresponding to the markings.
Each edge is thought of as glued from two half-edges and the legs also
count as half-edges.
Corresponding to $\Gamma$, there is a gluing map
\begin{equation*}
  \xi_\Gamma\colon \prod_{v \in \Gamma} \Mbar_{g(v), n(v)} \to \Mbar_{g, n},
\end{equation*}
finite of degree $|\Aut(\Gamma)|$, where $n(v)$ is the number of
half-edges at vertex $v$.
The $P_v$ in \eqref{eq:tautgen} are arbitrarily chosen monomials in
the first Chern classes $\psi_1, \dotsc, \psi_{n(v)}$ of the cotangent
line bundles at the markings and the $\kappa$-classes
\begin{equation*}
  \kappa_i = \pi_*(\psi_{n(v) + 1}^{i + 1}),
\end{equation*}
where $\pi$ is the forgetful map
\begin{equation*}
  \pi: \Mbar_{g(v), n(v) + 1} \to \Mbar_{g(v), n(v)}.
\end{equation*}
See \cite[Appendix~A]{GrPa03} for a review of these generators of the
tautological ring.
There, it is shown how multiplication, push-forward and pull-back
under the tautological maps can be defined on the level of generators.
In \cite{Pi12P}, the strata algebra $\strata_{g, n}$ is defined as the
formal algebra generated by formal classes of the form
\eqref{eq:tautgen}.
The discussion of \cite[Appendix~A]{GrPa03} defines the ring structure
on $\strata_{g, n}$ and lifts of the tautological maps under the
natural forgetful homomorphism
$q\colon \strata_{g, n} \to R^*(\Mbar_{g, n})$.

While there is an explicit set of generators, the set of relations
between the generators is not known despite much study.
The conjectures of Faber \cite{Fa99, FaPa00b} would give a description
of the relations.
However, recently, Petersen and Tommasi \cite{PeTo14} gave
counter-examples to the Gorenstein conjecture, which is the only one
of Faber's conjectures that had remained unproven.
On the other hand, Pixton's set \cite{Pi12P} of generalized
Faber--Zagier relations gives a well-tested conjectural description
for the set of tautological relations.

We now give a brief description of Pixton's relations.
As for the original relations of Faber and Zagier in $R^*(M_g)$, the
main input to the formulation of the relations of Pixton are the
hypergeometric series
\begin{multline}
  \label{eq:AB}
  A(z) = \sum_{i = 0}^\infty \frac{(6i)!}{(3i)!(2i)!} \frac{(-z)^i}{1728^i} = 1 - \frac 5{144} z \pm\dotsb, \\
  B(z) = \sum_{i = 0}^\infty \frac{(6i)!}{(3i)!(2i)!} \frac{1 + 6i}{1 - 6i} \frac{(-z)^i}{1728^i} = 1 + \frac 7{144} z \mp\dotsb.
\end{multline}
The functions $A$ and $B$ have appeared in many different ways in the
study of the moduli space of curves and are strongly related to the
asymptotic expansion of the Airy function.
See \cite{BJP15P} for a review of some of these occurrences.

Given a collection $(a_1, \dotsc, a_n) \in \{0, 1\}^n$, Pixton
constructs from these series certain elements
\begin{equation}
  \label{eq:pixrel}
  \sum_\Gamma \frac 1{|\Aut(\Gamma)|} \frac 1{2^{h_1(\Gamma)}} \xi_{\Gamma*} \left(\left[\prod_v \kappa_v \prod_i B_i \prod_e \Delta_e\right]_{\prod_v \zeta_v^{g(v) - 1}}\right)
\end{equation}
in the strata algebra $\strata_{g, n}$, where the sum is over all dual
graphs $\Gamma$ and the products are over all vertices $v$, markings
$i$ and edges $e$ of $\Gamma$.
At each vertex $v$ a parity condition is singled out using a variable
$\zeta_v$ satisfying $\zeta_v^2 = 1$.
The objects $\kappa_v$, $B_i$ and $\Delta_e$ in \eqref{eq:pixrel}
depend on the $\zeta$-variables and the square brackets single out a
specific coefficient in these variables.
Define
\begin{equation*}
  B_i =
  \begin{cases}
    A(\zeta_i \psi_i), & \text{if $a_i = 0$,} \\
    B(\zeta_i \psi_i), & \text{if $a_i = 1$,}
  \end{cases}
\end{equation*}
where $\zeta_i$ is $\zeta_v$ if the marking $i$ is at vertex $v$, and
define
\begin{equation*}
  \Delta_e = \frac{\zeta' + \zeta'' - A(\zeta'\psi')\zeta''B(\zeta''\psi'') - \zeta'B(\zeta'\psi')A(\zeta''\psi'')}{\psi' + \psi''},
\end{equation*}
where $\psi'$ and $\psi''$ are the cotangent line classes at both
sides of the node corresponding to $e$, and $\zeta'$ and $\zeta''$ are
the parity variables at the vertices connected by $e$.
Finally, $\kappa_v$ is the polynomial in formal $\kappa$-classes of
$\Mbar_{g(v), n(v)}$ defined by
\begin{equation*}
  \kappa_v = \sum_{k = 0}^\infty \pi_*\left(\prod_{i = n(v) + 1}^{n(v) + k} (\psi_i - \psi_i A(\zeta_v \psi_i))\right),
\end{equation*}
where $\pi$ forgets the last $k$ markings, via Faber's formula
\cite{ArCo96}.

In this paper we give a proof of the following result, which has been
conjectured by Pixton in \cite{Pi12P}.
\begin{thm}
  \label{thm:main}
  Let $(g, n)$ be in the stable range $2g - 2 + n > 0$, and let $d \ge
  0$ satisfy the condition
  \begin{equation}
    \label{eq:ineq}
    3d > g - 1 + \sum_{i = 1}^n a_i.
  \end{equation}
  Then the degree $d$ part of \eqref{eq:pixrel} is a tautological
  relation, that is, lies in the kernel of the forgetful map
  $q\colon \strata_{g, n} \to R^*(\Mbar_{g, n})$.
\end{thm}
Pixton furthermore conjectures that the smallest ideal in
$\strata_{g, n}$ including the relations of Theorem~\ref{thm:main} and
stable under push-forward under the tautological maps coincides with
the set of all tautological relations.

In \cite{PPZ15}, Theorem~\ref{thm:main} is proven on the level of
cohomology; in other words, it is proven that \eqref{eq:pixrel} is
mapped to zero under the composition of $q$ and the cycle class map.
The main ingredient of the proof of \cite{PPZ15} is the
Givental--Teleman classification of semisimple cohomological field
theories (CohFTs) applied in the example of Witten's 3-spin class.
All arguments of \cite{PPZ15} also apply to Chow rings except for the
fact that the classification cannot be used since its proof by Teleman
uses topological arguments in an essential way.

For the proof of Theorem~\ref{thm:main}, we will use the same
arguments as in \cite{PPZ15} except that we use another CohFT, namely
the equivariant Gromov--Witten theory of the projective line.
In this example, the classification is well known to be valid also in
the Chow setting by a localization calculation \cite{Gi01a} of
Givental.
The CohFT is slightly more complicated than the CohFT of Witten's
3-spin class in that it depends on one additional parameter.
We will, however, see that under a particular specialization of the
parameters, it agrees with the CohFT of Witten's 3-spin class.

The CohFT that we use can also alternatively be defined using the
moduli space of stable quotients to $\PP^1$ \cite{MOP11}.
Virtual localization for this moduli space has been used in
\cite{PaPi13P} to give the first proof of the Faber--Zagier relations
of $R^*(M_g)$.
The current work can be viewed as a simplification of the author's
first proof \cite{Ja13P} of Theorem~\ref{thm:main} generalizing the
arguments of \cite{PaPi13P}.
We prefer to use Gromov--Witten theory over stable quotients here
because the reconstruction of the CohFT in the Chow ring via
localization had already essentially been proven by Givental
\cite{Gi01a}.

Instead of $\mathbb P^1$, we could also have studied the
Gromov--Witten theory of higher-dimensional projective spaces in order
to find relations in the tautological ring.
The resulting relations become increasingly complicated with rising
dimension.
Still, in \cite{Ja15P1} it is shown that all of them can be expressed
in terms of Pixton's relations.

\subsection*{Plan of the paper}

We start by reviewing standard material on CohFTs including the
Givental--Teleman reconstruction in Section~\ref{sec:cohft}.
In Section~\ref{sec:p1} we restrict ourselves to the example of the
equivariant Gromov--Witten theory of $\PP^1$, studying first the
Frobenius algebra underlying the theory in Section~\ref{sec:p1:TQFT}
and the $R$-matrix in Section~\ref{sec:p1:R}.
In Section~\ref{sec:p1:limit}, we observe that under a particular
specialization of parameters this data coincides with the 3-spin
theory.
We use this in Section~\ref{sec:p1:rels} to conclude
Theorem~\ref{thm:main}.

In the appendix, we make explicit Givental's proof of
the reconstruction of the Gromov--Witten theory of projective spaces
and we recall his mirror symmetry description of the small theories.

\subsection*{Acknowledgments}

I thank my PhD advisor R.~Pandharipande for all his support while I
worked on this project, especially for the frequent meetings at the
beginning of this project.

I am also very grateful for comments from him and the anonymous
referee, which led to a considerable improvement of the current
text.

Special thanks are due to Y.~P.~Lee for discussions at the
conference \emph{Cohomology of the moduli space of curves} organized
by the \emph{Forschungsinstitut für Mathematik} at ETH Zürich, which
changed my way of thinking about the stable quotient relations.
I am also grateful for relevant discussions with A.~Marian,
D.~Oprea, A.~Pixton and D.~Zvonkine at various points of time.

The author was supported by the Swiss National Science Foundation
grant SNF 200021\_143274.

\section{Cohomological Field Theories}
\label{sec:cohft}

\subsection{Definitions}
\label{sec:cohft:def}

Cohomological field theories were first introduced by Kontsevich and
Manin in \cite{KoMa97} to formalize the structure of classes from
Gromov--Witten theory.
Let $V$ be an $N$-dimensional $\CC$-vector space and $\eta$ a
nonsingular bilinear form on $V$.
\begin{defi}
  A cohomological field theory (CohFT) $\Omega$ on $(V, \eta)$ (on the
  level of the Chow ring) is a system
  \begin{equation*}
    \Omega_{g, n} \in A^*(\Mbar_{g, n}) \otimes (V^*)^{\otimes n}
  \end{equation*}
  of multilinear forms with values in the Chow ring of $\Mbar_{g, n}$
  satisfying the following properties:
  \begin{description}
  \item[$S_n$-equivariance] The multilinear form $\Omega_{g, n}$ is
    invariant with respect to the $S_n$-action permuting the factors
    of $(V^*)^{\otimes n}$ and the markings of $\Mbar_{g, n}$
    simultaneously.
  \item[Gluing] The pull-back of $\Omega_{g, n}$ via the gluing map
    \begin{equation*}
      \Mbar_{g_1, n_1 + 1} \times \Mbar_{g_2, n_2 + 1} \to \Mbar_{g, n}
    \end{equation*}
    is given by the direct product of $\Omega_{g_1, n_2 + 1}$ and
    $\Omega_{g_2, n_2 + 1}$ with the bivector $\eta^{-1}$ inserted at
    the two points glued together.
    Similarly for the gluing map
    $\Mbar_{g - 1, n + 2} \to \Mbar_{g, n}$ the pull-back of
    $\Omega_{g, n}$ is given by $\Omega_{g - 1, n + 2}$ with
    $\eta^{-1}$ inserted at the two points glued together.
  \item[Unit] There is a special element $\mathbf 1 \in V$ called the
    \emph{unit} such that
    \begin{equation*}
      \Omega_{g, n + 1}(v_1, \dotsc, v_n, \mathbf 1)
    \end{equation*}
    is the pull-back of $\Omega_{g, n}(v_1, \dotsc, v_n)$ under the
    forgetful map and
    \begin{equation*}
      \Omega_{0, 3}(v, w, \mathbf 1) = \eta(v, w).
    \end{equation*}
  \end{description}
\end{defi}
\begin{defi}
  The quantum product $(u, v) \mapsto uv$ on $V$ with unit $\mathbf 1$
  is defined by the condition
  \begin{equation}
    \label{eq:defquantumproduct}
    \eta(uv, w) = \Omega_{0, 3} (u, v, w).
  \end{equation}
\end{defi}
\begin{defi}
  A CohFT is called semisimple if the algebra $V$ is semisimple,
  that is, if it admits a basis of orthogonal idempotent elements.
\end{defi}
\begin{remark}
  Many CohFTs come in families.
  The definition of a CohFT directly generalizes to the setting where
  $V$ is a module over a commutative ring $T$ by demanding
  $T$-linearity instead of $\CC$-linearity.
\end{remark}

\subsection{Examples}
\label{sec:cohft:ex}

\begin{ex}
  \label{ex:trivial}
  For each Frobenius algebra there is the \emph{trivial CohFT} (also
  called \emph{topological field theory} or TQFT) $\Omega_{g, n}$
  characterized by \eqref{eq:defquantumproduct} and the condition
  \begin{equation*}
    \Omega_{g, n} \in A^0 (\Mbar_{g, n}) \otimes (V^*)^{\otimes n}.
  \end{equation*}
  Let us record an explicit formula for the appendix:
  In the case that the Frobenius algebra is semisimple, there is a
  basis $\epsilon_i$ of orthogonal idempotents of $V$ and
  \begin{equation*}
    \tilde\epsilon_i = \frac{\epsilon_i}{\sqrt{\Delta_i}},
  \end{equation*}
  where $\Delta_i = \eta(\epsilon_i, \epsilon_i)$, is a corresponding
  orthonormal basis of \emph{normalized idempotents}.
  It is not difficult to show
  \begin{equation*}
    \Omega_{g, n}(\tilde\epsilon_{i_1}, \dotsc, \tilde\epsilon_{i_n}) =
    \begin{cases}
      \sum_j \Delta_{i_j}^{g - 1}, & \text{if } n = 0, \\
      \Delta_{i_1}^{\frac{2g - 2 + n}2}, & \text{if } i_1 = \dotsb = i_n, \\
      0, & \text{else,}
    \end{cases}
  \end{equation*}
  since, in general, $\Omega_{0, 3}$ is uniquely determined by $\eta$
  and the quantum product, and the other $\Omega_{g, n}$ are
  determined by considering the restriction to the locus of curves
  glued together from $2g - 2 + n$ rational curves with three special
  points.
\end{ex}

\begin{ex}
  \label{ex:hodge}
  The Chern polynomial $c_t(\mathbb E)$ of the Hodge bundle
  $\mathbb E$ gives a one-dimensional CohFT over the ring $\QQ[t]$.
\end{ex}

\begin{ex}
  \label{ex:GW}
  Let $X$ be a smooth, projective variety such that the cycle class
  map gives an isomorphism between Chow and cohomology rings.
  Then the Gromov--Witten theory of $X$ defines a CohFT based on the
  module $A^*(X) \otimes N$ over the Novikov ring
  $N := \CC[\![q^\beta]\!]$, which is generated by formal variables
  $q^\beta$ indexed by effective, integral curve classes $\beta$. 
  The CohFT is defined by setting, for $v_1, \dotsc, v_n \in A^*(X)$,
  \begin{equation*}
    \Omega_{g, n} (v_1, \dotsc, v_n)
    = \sum_{\beta} p_*\left(\prod_{i = 1}^n \ev_i^*(v_i) \cap [\Mbar_{g, n}(X, \beta)]^\vir\right) q^\beta,
  \end{equation*}
  where the sum ranges over effective, integral curve classes, $\ev_i$
  is the $i$th evaluation map and $p$ is the forgetful map
  $p\colon \Mbar_{g, n}(X, \beta) \to \Mbar_{g, n}$.
  The gluing property follows from the splitting axiom of virtual
  fundamental classes.
  The fundamental class of $X$ is the unit of the CohFT and the unit
  axioms follow from the identity axiom in Gromov--Witten theory.

  For a torus action on $X$, this example can be enhanced to give a
  CohFT from the $T$-equivariant Gromov--Witten theory of $X$.
  In this case, the CohFT is defined on the module
  $A_T^*(X) \otimes N$.
  The case that $X = \PP^1$ with a $\CC^*$-action will be the main
  example we will study.
  We will make the definition of the CohFT more concrete in
  Section~\ref{sec:p1:defs}.
\end{ex}

\subsection{Reconstruction}
\label{sec:cohft:reconstr}

The (upper half of the) \emph{symplectic loop group} corresponding to
a vector space $V$ with nonsingular bilinear form $\eta$ is the group
of endomorphism-valued power series $V[\![z]\!]$ such that the
\emph{symplectic condition}
\begin{equation}
  \label{eq:sympl}
  R(z) R^t(-z) = 1
\end{equation}
holds.
Here $R^t$ is the adjoint of $R$ with respect to $\eta$.
There is an action of this group on the space of all CohFTs based on a
fixed semisimple Frobenius algebra structure of $V$.
The action is often named after Givental because he introduced it on
the level of arbitrary genus Gromov--Witten potentials.

Given a CohFT $\Omega_{g, n}$ and such an endomorphism $R$, the new
CohFT $R\Omega_{g, n}$ takes the form of a sum over dual graphs
$\Gamma$
\begin{equation}
  \label{eq:reconstr}
  R\Omega_{g, n}(v_1, \dotsc, v_n)
  = \sum_\Gamma \frac 1{\Aut(\Gamma)} \xi_*\left(\prod_v \sum_{k = 0}^\infty \frac 1{k!} \pi_* \Omega_{g(v), n(v) + k}(\dots)\right),
\end{equation}
where $\xi\colon \prod_v \Mbar_{g(v), n(v)} \to \Mbar_{g, n}$ is the
gluing map of curves of topological type $\Gamma$ from their
irreducible components,
$\pi\colon \Mbar_{g(v), n(v) + k} \to \Mbar_{g(v), n(v)}$ forgets the
last $k$ markings and we still need to specify what is put into the
arguments of $\prod_v \Omega_{g(v), n(v) + k}$.
Instead of allowing only vectors in $V$ to be put into
$\Omega_{g, n}$, we will allow elements of
$V[\![\psi_1, \dotsc, \psi_n]\!]$, where $\psi_i$ acts on the
cohomology of the relevant moduli space of curves by multiplication
with the $i$th cotangent line class.
\begin{itemize}
\item Into each argument corresponding to a marking of the curve, put
  $R^{-1}(\psi)$ applied to the corresponding vector.
\item Into each pair of arguments corresponding to an edge put the
  bivector
  \begin{equation*}
    \frac{R^{-1}(\psi_1) \otimes R^{-1}(\psi_2) - \Id \otimes \Id}{-\psi_1 - \psi_2} \eta^{-1} \in V^{\otimes 2}[\![\psi_1, \psi_2]\!],
  \end{equation*}
  where one has to substitute the $\psi$-classes at each side of the
  normalization of the node for $\psi_1$ and $\psi_2$.
  By the symplectic condition this is well-defined.
\item At each of the additional arguments for each vertex put
  \begin{equation*}
    T(\psi) := \psi(\Id - R^{-1}(\psi)) \mathbf 1,
  \end{equation*}
  where $\psi$ is the cotangent line class corresponding to that
  vertex.
  Since $T(z) = O(z^2)$ the above $k$-sum is finite.
\end{itemize}

The following reconstruction result (on the level of potentials) was
first proposed by Givental \cite{Gi01b}.
\begin{thm}[\cite{Te12}]
  \label{thm:class}
  In cohomology, the $R$-matrix action is free and transitive on the
  space of semisimple CohFTs based on a given Frobenius algebra.
  In particular, for any CohFT $\Omega$ with underlying TQFT $\omega$
  there exists an $R$-matrix $R$ such that $\Omega = R\omega$.
\end{thm}
\begin{remark}
  Teleman's proof relies heavily on topological results (Mumford's
  conjecture/Madsen-Weiss theorem) and it is therefore not known if
  the same classification result also holds in general when we work
  with Chow rings instead of cohomology.
\end{remark}
\begin{remark}
  \label{rmk:reconstrChow}
  In the appendix, we review a virtual localization computation of
  Givental, which implies that Theorem~\ref{thm:class} is true in the
  Chow ring for the equivariant Gromov--Witten theory of projective
  spaces.
  Furthermore, mirror symmetry gives an explicit description of the
  $R$-matrix.
\end{remark}
\begin{ex}
  \label{ex:mumford}
  By Mumford's Grothendieck--Riemann--Roch calculation \cite{Mu83},
  the single entry of the $R$-matrix taking the trivial
  one-dimensional CohFT to the CohFT from Example~\ref{ex:hodge} is
  given by
  \begin{equation*}
    \exp\left(\sum_{i = 1}^\infty \frac{B_{2i}}{2i(2i - 1)} (tz)^{2i - 1}\right),
  \end{equation*}
  where the $B_{2i}$ are the Bernoulli numbers, defined by
  \begin{equation*}
    \sum_{i = 0}^\infty B_i \frac{x^i}{i!} = \frac x{e^x - 1}.
  \end{equation*}
  More generally, if we consider a more general CohFT given by a
  product of Chern polynomials (in different variables) of the Hodge
  bundle, the $R$-matrix from the trivial CohFT is the product of the
  $R$-matrices of the factors.
\end{ex}

\section{Equivariant Gromov--Witten theory of $\PP^1$}
\label{sec:p1}

In this section we consider the projective line $\PP^1$ together with
a $\CC^*$-action with weights $(0, 1)$ and the tautological relations
resulting from it.

\subsection{Definition of the CohFT}
\label{sec:p1:defs}

We make the CohFT described in Example~\ref{ex:GW} more concrete in
our example.

The only effective curve classes on $\PP^1$ are nonnegative multiples
of $[\PP^1]$.
Hence the Novikov ring is the ring $\CC[\![q]\!]$, where $q$
corresponds to $[\PP^1]$.

Let $\lambda$ be the first Chern class of the dual of the one
dimensional vector space with a $\CC^*$-action of weight 1.
Then the equivariant Chow ring of $\PP^1$ is isomorphic to
\begin{equation*}
  A_{\CC^*}^*(\PP^1) \cong \QQ[\lambda, H]/H(H - \lambda);
\end{equation*}
this expresses the fact that the equivariant classes $H$ and
$H - \lambda$ of the two fixed points 0 and $\infty$ intersect
trivially.
The module $A_{\CC^*}^*(\PP^1) \otimes N$ over
$\CC[\lambda] \otimes N$ is the state space of the CohFT.

The CohFT $\Omega_{g, n}^{\PP^1}$ corresponding to the equivariant
Gromov--Witten theory of $\PP^1$ is defined by setting for
$v_1, \dotsc, v_n \in A_{\CC^*}^*(\PP^1)$,
\begin{equation}
  \label{eq:P1CohFTdef}
  \Omega_{g, n} (v_1, \dotsc, v_n)
  = \sum_{d = 0}^\infty p_*\left(\prod_{i = 1}^n \ev_i^*(v_i) \cap [\Mbar_{g, n}(\PP^1, d)]^\vir\right) q^d,
\end{equation}
where $\Mbar_{g, n}(\PP^1, d)$ is the moduli space of stable maps of
degree $d$ to $\PP^1$ and $p$ is the forgetful map to $\Mbar_{g, n}$.
The dimension of the virtual class $[\Mbar_{g, n}(\PP^1, d)]^\vir$ is
\begin{equation*}
  (3 - \dim\PP^1)(g - 1) + n + \langle dH, c_1(T_{\PP^1})\rangle = 2g - 2 + n + 2d.
\end{equation*}
This means that for any $a_1, \dotsc, a_n \in \{0, 1\}$ the
coefficient of $q^d$ in $\Omega_{g, n} (H^{a_1}, \dotsc, H^{a_n})$ has
degree
\begin{equation*}
  3g - 3 + n - (2g - 2 + n + 2d) + \sum_{i = 1}^n a_i = g - 1 - 2d + \sum_{i = 1}^n a_i.
\end{equation*}
Since this integer is negative for large enough $d$, the sum over $d$
in \eqref{eq:P1CohFTdef} is finite and defines a CohFT on the module
$A_{\CC^*}^*(\PP^1)[q]$ over the ring $\CC[\lambda, q]$.

\subsection{The underlying TQFT}
\label{sec:p1:TQFT}

The nonsingular bilinear form on $A_T^*(\PP^1)$ is given by the
equivariant Poincaré pairing.
In the basis $\{1, H\}$ it takes the form
\begin{equation*}
  \begin{pmatrix}
    0 & 1 \\
    1 & \lambda
  \end{pmatrix}.
\end{equation*}

Define
\begin{equation*}
  \tH := H - \frac\lambda 2,
\end{equation*}
so that in the basis $\{1, \tH\}$ the Poincaré pairing becomes
just
\begin{equation*}
  \begin{pmatrix}
    0 & 1 \\
    1 & 0
  \end{pmatrix}.
\end{equation*}

The equivariant quantum product of $\PP^1$ is defined using
three-pointed genus zero equivariant Gromov--Witten invariants as in
\eqref{eq:defquantumproduct}.
Explicitly, the relation defining the classical equivariant cup
product gets deformed to
\begin{equation*}
  H(H - \lambda) = q.
\end{equation*}
Equivalently, we can write
\begin{equation*}
  \tH^2 = \phi,
\end{equation*}
where we set
\begin{equation*}
  \phi = \frac{\lambda^2}4 + q.
\end{equation*}
Notice that the product is semisimple if and only if $\phi \neq 0$.

By choosing a root $\sqrt\phi$ of $\phi$ (working in an extension of
the base ring), we can easily write down idempotents
\begin{equation*}
  \frac 1{\Delta_\pm} \tH + \frac 12,
\end{equation*}
where $\Delta_\pm = \pm 2\sqrt\phi$.
By choosing further roots $\sqrt{\Delta_\pm}$, we can define
normalized idempotents such that the change of basis $\Psi$ from the
basis of normalized idempotents to the basis $\{1, \tH\}$ is given by
\begin{equation*}
  \Psi =
  \begin{pmatrix}
    \frac{\sqrt\phi}{\sqrt{\Delta_+}} & \frac{-\sqrt\phi}{\sqrt{\Delta_-}} \\
    \frac 1{\sqrt{\Delta_+}} & \frac 1{\sqrt{\Delta_-}}
  \end{pmatrix}.
\end{equation*}
For convenience, we also note that
\begin{equation*}
  \Psi^{-1} =
  \begin{pmatrix}
    \frac 1{\sqrt{\Delta_+}} & \frac{\sqrt\phi}{\sqrt{\Delta_+}} \\
    \frac 1{\sqrt{\Delta_-}} & \frac{-\sqrt\phi}{\sqrt{\Delta_-}}
  \end{pmatrix}.
\end{equation*}

By induction, we can show that the underlying TQFT $\omega_{g, n}$ of
$\Omega_{g, n}^{\PP^1}$ is explicitly given by
\begin{equation}
  \label{eq:TQFT}
  \omega_{g, n}(1^{\otimes a}, \tH^{\otimes b}) =
  \begin{cases}
    2^g \phi^{\frac{g - 1 + b}2}, & \text{if $g - 1 + b$ is even,} \\
    0, & \text{else,}
  \end{cases}
\end{equation}
where we have used a shorthand notation for the arguments of
$\omega_{g, n}$, meaning $a$ insertions of 1 and $b$ insertions of
$\tH$.
Notice that the TQFT applied to basis vectors takes values in
$\QQ[\phi]$.

\subsection{The $R$-matrix}
\label{sec:p1:R}

As reviewed in the appendix, the $R$-matrix necessary for
reconstructing the equivariant Gromov--Witten theory of $\PP^1$ can be
found by using virtual localization for a nontrivial $\CC^*$-action
and made explicit using stationary phase asymptotics of certain
oscillating integrals on the so-called mirror of $\PP^1$.
The coefficients of the $R$-matrix obtained from virtual localization
are valued in $\QQ[\lambda^\pm][\![q, z]\!]$, but the mirror
description shows that the coefficients are obtained by expanding
elements of $\QQ[\lambda^\pm, \phi^\pm][\![z]\!]$ around $q = 0$.
In any case, the $R$-matrix is defined over a nontrivial ring
extension of $\QQ[\lambda, q]$, and it is, at first sight, not clear
how this is compatible with the fact that the CohFT is defined over
$\CC[\lambda, q]$.
As we will see in Section~\ref{sec:p1:rels}, such a contrast implies
the existence of tautological relations.

As reviewed in Section~\ref{sec:mirror} of the appendix (see also
\cite{BJP15P}), the $R$-matrix can be computed from the stationary
phase asymptotics of the oscillating integrals
\begin{equation}
  \label{eq:osc}
  \frac 1{\sqrt{-2\pi z}} \int\limits_{\Gamma_\pm} e^{F(x)/z} \left(e^x - \frac\lambda 2\right)^{1 - i} \mathrm dx,
\end{equation}
where $\Gamma_\pm$ are suitable real, one-dimensional cycles, and
\begin{equation*}
  F(x) = e^x + qe^{-x} - \lambda x
\end{equation*}
is the superpotential.
Let $u_{\pm}$ be the critical values at the critical points
\begin{equation*}
  e^x = \frac \lambda 2 \pm \sqrt\phi.
\end{equation*}
of $F$.
The $\Gamma_\pm$ are the corresponding Lefschetz thimbles, each going
through exactly one of the critical points.

We start with the stationary phase asymptotics for the integral
\eqref{eq:osc} when $i = 1$.
Shifting each critical point to the origin, rescaling $x$, expanding
the exponential and using
\begin{equation*}
  \frac 1{\sqrt{2\pi}} \int\limits_{-\infty}^\infty e^{-\frac{x^2}2} x^i\mathrm dx =
  \begin{cases}
    (i - 1)!!, & \text{if $i$ is even,} \\
    0, & \text{else,}
  \end{cases}
\end{equation*}
gives
\begin{align*}
  &\frac 1{\sqrt{-2\pi z}} \int\limits_{\Gamma_\pm} e^{\left(u_\pm + \Delta_\pm \frac{x^2}2 + \lambda \frac{x^3}6 + \Delta_\pm \frac{x^4}{24} + \lambda \frac{x^5}{120} + \dotsb\right)/z} \mathrm dx \\
  \asymp& \frac{e^{\frac{u_\pm}z}}{\sqrt{\Delta_\pm} \sqrt{2\pi}} \int\limits_{-\infty}^\infty e^{-\frac{x^2}2} e^{-\frac{x^3}6 \frac\lambda{\Delta_\pm^{3/2}} \sqrt{-z} - \frac{x^4}{24} \frac 1{\Delta_\pm} (-z) - \frac{x^5}{120} \frac\lambda{\Delta_\pm^{5/2}} (-z)^{3/2} - \dotsb} \mathrm dx \\
  \asymp& \frac{e^{\frac{u_\pm}z}}{\sqrt{\Delta_\pm}} F_0\left(z\frac{\lambda^2}{\Delta_\pm^3}, \frac{\Delta_\pm^2}{\lambda^2}\right)
\end{align*}
for some power series
\begin{equation*}
  F_0(x, y) = 1 - \left(\frac 5{24} - \frac 18 y\right)x + \left(\frac{385}{1152} - \frac{77}{192} y + \frac 9{128} y^2\right) x^2 + \dotsb \in \QQ[y][\![x]\!].
\end{equation*}

When $i = 0$, the additional factor
\begin{equation*}
  e^x - \frac\lambda 2
\end{equation*}
in \eqref{eq:osc} becomes
\begin{equation*}
  \left(\frac\lambda 2 \pm \sqrt\phi\right)e^x - \frac\lambda 2
\end{equation*}
after translation and
\begin{equation*}
  \pm\sqrt\phi \left(\left(1 + \frac\lambda{\Delta_\pm}\right)e^{x\Delta_\pm^{-1/2}\sqrt{-z}} - \frac\lambda{\Delta_\pm}\right).
\end{equation*}
after scaling.
So for $i = 0$, the asymptotic expansion is given by
\begin{equation*}
  \frac{\pm\sqrt\phi e^{\frac{u_\pm}z}}{\sqrt{\Delta_\pm}} F_1\left(z\frac{\lambda^2}{\Delta_\pm^3}, \frac{\Delta_\pm^2}{\lambda^2}\right),
\end{equation*}
where
\begin{equation*}
  F_1(x, y) = 1 + \left(\frac 7{24} - \frac 38 y\right)x - \left(\frac{455}{1152} - \frac{33}{64} y + \frac{15}{128} y^2\right) x^2 + \dotsb \in \QQ[y][\![x]\!].
\end{equation*}

In total, the $S$-matrix from the basis of normalized idempotents to
the basis $\{1, \tH\}$ is given by
\begin{equation*}
  \begin{pmatrix}
    \frac{\sqrt\phi}{\sqrt{\Delta_+}} F_1\left(z\frac{\lambda^2}{\Delta_+^3}, \frac{\Delta_+^2}{\lambda^2}\right) & \frac{-\sqrt\phi}{\sqrt{\Delta_-}} F_1\left(z\frac{\lambda^2}{\Delta_-^3}, \frac{\Delta_-^2}{\lambda^2}\right) \\
    \frac 1{\sqrt{\Delta_+}} F_0\left(z\frac{\lambda^2}{\Delta_+^3}, \frac{\Delta_+^2}{\lambda^2}\right) & \frac 1{\sqrt{\Delta_-}} F_0\left(z\frac{\lambda^2}{\Delta_-^3}, \frac{\Delta_-^2}{\lambda^2}\right)
  \end{pmatrix}
  \begin{pmatrix}
    e^{\frac{u_+}z} & 0 \\
    0 & e^{\frac{u_-}z}
  \end{pmatrix}.
\end{equation*}
Multiplying the first factor from the right by $\Psi^{-1}$ gives the
$R$-matrix written in the basis $\{1, \tH\}$:
\begin{equation}
  \label{eq:R}
  R(z) =
  \begin{pmatrix}
    F_1^e & \sqrt\phi F_1^o \\
    \frac 1{\sqrt\phi} F_0^o & F_0^e
  \end{pmatrix}
  \left(z\frac{\lambda^2}{\Delta_+^3}, \frac{\Delta_+^2}{\lambda^2}\right)
  =
  \begin{pmatrix}
    F_1^e & \sqrt\phi F_1^o \\
    \frac 1{\sqrt\phi} F_0^o & F_0^e
  \end{pmatrix}
  \left(z\frac{\lambda^2}{8 \phi^{3/2}}, 4\frac\phi{\lambda^2}\right)
  ,
\end{equation}
where
\begin{equation*}
  F_i^e(x, y) = \frac{F_i(x, y) + F_i(-x, y)}2, \qquad F_i^o(x, y) = \frac{F_i(x, y) - F_i(-x, y)}2.
\end{equation*}
The symplectic condition \eqref{eq:sympl} implies
\begin{equation}
  \label{eq:sympl2}
  R^{-1}(z) =
  \begin{pmatrix}
    0 & 1 \\
    1 & 0
  \end{pmatrix}
  R^t(-z)
  \begin{pmatrix}
    0 & 1 \\
    1 & 0
  \end{pmatrix}.
\end{equation}
So we can easily compute
\begin{equation*}
  R^{-1}(z) =
  \begin{pmatrix}
    F_0^e & \sqrt\phi F_1^o \\
    \frac 1{\sqrt\phi} F_0^o & F_1^e
  \end{pmatrix}
  \left(-z\frac{\lambda^2}{8 \phi^{3/2}}, 4\frac\phi{\lambda^2}\right).
\end{equation*}
\begin{remark}
  \label{rmk:ring}
  The power series $F_i^e(x, y)$ are even in $x$; the power series
  $F_i^o(x, y)$ are odd.
  In particular, the coefficients of $R(z)$ in the basis $\{1, \tH\}$
  are power series in $z$ whose coefficients are Laurent polynomials
  in $\lambda$ and $\phi$.
\end{remark}

\subsection{Limit}
\label{sec:p1:limit}

In this section, we show how, out of the CohFT
$\Omega_{g, n}^{\PP^1}$, a CohFT can be extracted that is very similar
to the theory of the shifted Witten's 3-spin class as studied in
\cite{PPZ15}.

We use the isomorphisms
\begin{equation*}
  \QQ[\lambda, q] \cong \QQ[\lambda, \phi], \qquad A_{\CC^*}^*(\PP^1) \otimes \QQ[q] \cong \QQ[\lambda, \phi]\langle 1, \tH\rangle,
\end{equation*}
defined by
\begin{equation*}
  q \mapsto \phi - \frac{\lambda^2}4, \qquad H \mapsto \tH + \frac\lambda 2,
\end{equation*}
once and for all so that the CohFT $\Omega_{g, n}^{\PP^1}$ is defined
over the ring $\QQ[\lambda, \phi]$ and based on the module
$\QQ[\lambda, \phi]\langle 1, \tH\rangle$.
(It follows from the discussion of Section~\ref{sec:p1:R} that we can
use $\QQ$-coefficients (instead of $\CC$-coefficients), but it is not
of much importance which field we use.)

We can define a new CohFT $\widetilde \Omega_{g, n}^{\PP^1}$ on the
module $\QQ[\lambda^{\pm}, \phi]\langle 1, \tH\rangle$ over the
ring $\QQ[\lambda^{\pm 1}, \phi]$ as the composition
\begin{equation*}
  \widetilde \Omega_{g, n}^{\PP^1} = \varphi \circ \Omega_{g, n}^{\PP^1},
\end{equation*}
where
\begin{equation*}
  \varphi\colon A^*(\Mbar_{g, n}) \otimes \QQ[\lambda, \phi] \to A^*(\Mbar_{g, n}) \otimes \QQ[\lambda^\pm, \phi]
\end{equation*}
is the ring homomorphism induced by
\begin{equation*}
  C \mapsto \left(-\frac 4{3\lambda^2}\right)^i C
\end{equation*}
for any $C \in A^i(\Mbar_{g, n})$.
By definition, $\varphi$ leaves elements of
$A^0(\Mbar_{g, n}) \otimes \QQ[\lambda, \phi]$ invariant and therefore
the TQFTs of $\Omega_{g, n}^{\PP^1}$ and
$\widetilde\Omega_{g, n}^{\PP^1}$ are the same (up to the change of
base ring). 
Since the formal variable $z$ of an $R$-matrix measures degree in the
Chow ring, we have
\begin{equation}
  \label{eq:tildeR}
  \widetilde R(z) = R\left(-\frac 43 z \lambda^{-2}\right)
  =
  \begin{pmatrix}
    F_1^e & \sqrt\phi F_1^o \\
    \frac 1{\sqrt\phi} F_0^o & F_0^e
  \end{pmatrix}
  \left(-\frac z{6 \phi^{3/2}}, 4\frac\phi{\lambda^2}\right),
\end{equation}
where $R(z)$ and $\widetilde R(z)$ denote the $R$-matrices of
$\Omega_{g, n}^{\PP^1}$ and $\widetilde \Omega_{g, n}^{\PP^1}$,
respectively.

Notice that no positive power of $\lambda$ appears in the TQFT
\eqref{eq:TQFT} and in $\widetilde R$.
So by the reconstruction, which also holds for
$\widetilde \Omega_{g, n}^{\PP^1}$ in the Chow ring, the base ring for
the CohFT $\widetilde \Omega_{g, n}^{\PP^1}$ can be taken to be the
intersection
\begin{equation*}
  \QQ[\lambda^{\pm}, \phi] \cap \QQ[\lambda^{-1}, \phi^{\pm}] = \QQ[\lambda^{-1}, \phi].
\end{equation*}
We define a new CohFT $\Omega_{g, n}$ on the module
$\QQ[\phi]\langle 1, \tH\rangle$ over the base ring $\QQ[\phi]$ by
setting $\lambda^{-1} = 0$ in $\widetilde \Omega_{g, n}^{\PP^1}$.
Reconstruction in the Chow ring for $\widetilde \Omega_{g, n}^{\PP^1}$
implies reconstruction in the Chow ring for $\Omega_{g, n}$.

The CohFT $\Omega_{g, n}$ has the same underlying TQFT \eqref{eq:TQFT}
as before but the simpler $R$-matrix
\begin{equation*}
  \begin{pmatrix}
    F_1^e & \sqrt\phi F_1^o \\
    \frac 1{\sqrt\phi} F_0^o & F_0^e
  \end{pmatrix}
  \left(-\frac z{6\phi^{3/2}}, 0\right)
  =
  \begin{pmatrix}
    F_1^e & -\sqrt\phi F_1^o \\
    -\frac 1{\sqrt\phi} F_0^o & F_0^e
  \end{pmatrix}
  \left(\frac z{6\phi^{3/2}}, 0\right).
\end{equation*}
From their definitions, we can make the $F_i(z, 0)$ explicit:
\begin{align*}
  F_0(z, 0) &\asymp \frac 1{\sqrt{2\pi}} \int\limits_{-\infty}^\infty e^{-\frac{x^2}2} e^{-\frac{x^3}6 \sqrt{-z}} \mathrm dx
  \asymp \sum_{i = 0}^\infty \frac{(6i - 1)!!}{(2i)!} \frac{(-z)^i}{36^i}
  &= A(6z) \\
  F_1(z, 0) &\asymp \frac 1{\sqrt{2\pi}} \int\limits_{-\infty}^\infty e^{-\frac{x^2}2} e^{-\frac{x^3}6 \sqrt{-z}} (1 + x\sqrt{-z}) \mathrm dx
  \asymp \dotsb
  &= B(6z)
\end{align*}
Here, the Faber--Zagier $A$- and $B$-series \eqref{eq:AB} appear.
Using the symplectic condition \eqref{eq:sympl2}, we can also compute
the inverse of the $R$-matrix to be equal to
\begin{equation}
  \label{eq:Rinv}
  \begin{pmatrix}
    A^e & \sqrt\phi B^o \\
    \frac 1{\sqrt\phi} A^o & B^e
  \end{pmatrix}
  \left(\frac z{\phi^{3/2}}\right),
\end{equation}
where $A^e$ and $B^e$ are the even parts of $A$ and $B$, respectively,
and similarly $A^o$ and $B^o$ denote the odd parts.
The coefficients of the inverse $R$-matrix are elements of
$\QQ[\phi^\pm][\![q]\!]$.

In cohomology, the CohFT of the shifted Witten's 3-spin class as
described in \cite{PPZ15} coincides with $\Omega_{g, n}$ if we
identify the parameters $\phi$ here and in \cite{PPZ15} as well as the
bases $\{1, \tH\}$ and $\{\partial_x, \partial_y\}$.
This is because the TQFT \eqref{eq:TQFT} (compare to
\cite[Lemma~3.3]{PPZ15}) and inverse $R$-matrix \eqref{eq:Rinv}
(compare to \cite[Equation~(14)]{PPZ15}) completely agree.\footnote{Note that the TQFT and inverse $R$-matrix in \cite{PPZ15}
  are written in the rescaled basis
  $\{\hat\partial_x, \hat\partial_y\}$, where
  $\hat\partial_x = \phi^{1/4}\partial_x$ and
  $\hat\partial_y = \phi^{-1/4}\partial_y$.} 
It is not clear whether $\Omega_{g, n}$ agrees with Witten's 3-spin
class in the Chow ring since it is not known whether reconstruction in
the Chow ring holds for the 3-spin theory.

\subsection{Relations}
\label{sec:p1:rels}

Recall from the previous section that there exists a CohFT
$\Omega_{g, n}$ defined on the $\QQ[\phi]$-module
$\QQ[\phi]\langle 1, \tH\rangle$, which coincides with the CohFT of
shifted Witten's 3-spin class in cohomology, but for which
reconstruction holds in the Chow ring.
Let us fix a choice of $a_1, \dotsc, a_n \in \{0, 1\}$ for this
section.

The fact that the inverse $R$-matrix \eqref{eq:Rinv} has poles in
$\phi$ implies the existence of tautological relations: The
reconstruction formula \eqref{eq:reconstr} for the CohFT
$\Omega_{g, n}$ defines an element
\begin{equation*}
  \cR_{g, n}(\tH^{a_1}, \dotsc, \tH^{a_n}) \in \strata_{g, n} \otimes \QQ[\phi^\pm],
\end{equation*}
projecting to $\Omega_{g, n}(\tH^{a_1}, \dotsc, \tH^{a_n})$ under
$q\colon \strata_{g, n} \to R^*(\Mbar_{g, n})$.
So all coefficients of $\phi^c$ in $\cR_{g, n}$ for $c < 0$ have to
project to zero and thus are tautological relations.
As we will see in this section these relations turn out to be
nontrivial and are in fact enough to prove Theorem~\ref{thm:main}.

\begin{table}
  \label{tab:degs}
  \centering
  \begin{tabular}{c|c|c|c|c|}
    & $\lambda$ & $q$ & $\phi$ & $C \in A^i(\Mbar_{g, n})$ \\ \hline
    $\Omega_{g, n}^{\PP^1}(\tH^{a_1}, \dotsc, \tH^{a_n})$ & 1 & 2 & 2 & i \\
    $\widetilde\Omega_{g, n}^{\PP^1}(\tH^{a_1}, \dotsc, \tH^{a_n})$ & 1 & 2 & 2 & 3i \\
    $\Omega_{g, n}(\tH^{a_1}, \dotsc, \tH^{a_n})$ & & & 2 & 3i \\
  \end{tabular}
  \caption{Degree rules}
\end{table}
We can characterize the part of $\cR_{g, n}$ with poles in $\phi$ in a
different way: By the dimension analysis from
Section~\ref{sec:p1:defs}, the coefficient of $\lambda^j q^d$ in
$\Omega_{g, n}^{\PP^1}(\tH^{a_1}, \dotsc, \tH^{a_n})$ is an element of
$A^k(\Mbar_{g, n})$ for
\begin{equation*}
  k = g - 1 - 2d - j + \sum_{i = 1}^n a_i.
\end{equation*}
Hence, $\Omega_{g, n}^{\PP^1}(\tH^{a_1}, \dotsc, \tH^{a_n})$ is
homogeneous of degree
\begin{equation}
  \label{eq:relvdim}
  g - 1 + \sum_{i = 1}^n a_i
\end{equation}
if we assign degrees according to the first row of
Table~\ref{tab:degs}.

Recall from Section~\ref{sec:p1:limit} that
$\widetilde\Omega_{g, n}^{\PP^1} = \varphi \circ \Omega_{g,
  n}^{\PP^1}$, where $\varphi$ replaces each $C \in A^i(\Mbar_{g, n})$
by a constant multiple of $C\lambda^{-2i}$.
Therefore
$\widetilde\Omega_{g, n}^{\PP^1}(\tH^{a_1}, \dotsc, \tH^{a_n})$ is
also homogeneous of degree \eqref{eq:relvdim}, if we assign degrees as
in the second row of Table~\ref{tab:degs}.
Finally, $\Omega_{g, n}(\tH^{a_1}, \dotsc, \tH^{a_n})$ is homogeneous
of the same degree \eqref{eq:relvdim} if we assign degrees as in the
third row of Table~\ref{tab:degs}.

The conclusion is that the part of
\begin{equation}
  \label{eq:CohFTrel}
  \cR_{g, n}(\tH^{a_1}, \dotsc, \tH^{a_n})
\end{equation}
with poles in $\phi$ corresponds to the part whose degree in the Chow
ring is greater than
\begin{equation}
  \label{eq:degwitten}
  \frac{g - 1 + \sum_{i = 1}^n a_i}3.
\end{equation}
This is exactly the degree of Witten's 3-spin class.
In fact, as discussed before, in cohomology the CohFT $\Omega_{g, n}$
is the same as the CohFT of Witten's 3-spin class and, as in
\cite{PPZ15}, the considered relations are those of degree greater
than the degree \eqref{eq:degwitten} of Witten's class.
So the relations we consider are the same elements of the strata
algebra as in \cite{PPZ15}.
Starting from the fact that the Givental--Teleman classification of
semisimple CohFTs holds in the Chow ring for the CohFT
$\Omega_{g, n}^{\PP^1}$ we have shown that the relations considered in
\cite{PPZ15} hold not only in cohomology but also in the Chow ring.

In \cite{PPZ15}, the resulting relations are formally simplified to
give the relations of Theorem~\ref{thm:main}.
Thus, we can conclude the proof of Theorem~\ref{thm:main} here.
For the convenience of the reader we provide a summary of the
arguments below.

\subsection{Computing the relations}

Following \cite{PPZ15}, we indicate how the relations from the
vanishing of \eqref{eq:CohFTrel} in degree greater than
\eqref{eq:degwitten} imply the relations of Theorem~\ref{thm:main}.
For simplicity, let us set $\phi = 1$. 
We expand the reconstruction formula \eqref{eq:reconstr} defining
\eqref{eq:CohFTrel} using the explicit inverse $R$-matrix
\eqref{eq:Rinv} and the formula \eqref{eq:TQFT} for the TQFT.
The terms which give a relation are those of degree greater than
\eqref{eq:degwitten}, which explains inequality \eqref{eq:ineq}.
For a dual graph $\Gamma$, the powers of $2$ in \eqref{eq:TQFT} yield
\begin{equation*}
  \prod_v 2^{g(v)} = 2^g 2^{-h^1(\Gamma)}.
\end{equation*}
The edge term in the reconstruction formula is
\begin{multline*}
  \frac{R^{-1}(\psi') \otimes R^{-1}(\psi'') - \Id \otimes \Id}{-\psi' - \psi''} \eta^{-1} \\
  = \frac{1 \otimes \tH + \tH \otimes 1 - (A^e(\psi') 1 + A^o(\psi') \tH) \otimes (B^o(\psi'') 1 + B^e(\psi'') \tH)}{\psi' + \psi''} \\
  - \frac{(B^o(\psi') 1 + B^e(\psi') \tH) \otimes (A^e(\psi'') 1 + A^o(\psi'') \tH)}{\psi' + \psi''} \\
  = [\Delta_e]_{\zeta'^0\zeta''^0} (1 \otimes 1) + [\Delta_e]_{\zeta'^1\zeta''^0} (\tH \otimes 1) + [\Delta_e]_{\zeta'^0\zeta''^1} (1 \otimes \tH) + [\Delta_e]_{\zeta'^1\zeta''^1} (\tH \otimes \tH),
\end{multline*}
where $[\Delta_e]_{\zeta'^i\zeta''^j}$ is the coefficient of
$\zeta'^i\zeta''^j$ in the edge series $\Delta_e$ from the
introduction.
Similarly, in the reconstruction formula, the contribution of the
$i$th marking and of an extra marking can be expressed in terms of
$B_i$ and $\kappa_v$, respectively.
As necessitated by the case distinction in \eqref{eq:TQFT}, we can
keep track of the parity of the number of $\tH$-insertions into the
TQFT at vertex $v$ using an additional variable $\zeta_v$.
By careful inspection, we see that \eqref{eq:CohFTrel} at $\phi = 1$
coincides with \eqref{eq:pixrel} times $2^g$.
We have thus arrived at the statement of Theorem~\ref{thm:main}.

\appendix

\section{Virtual localization for projective spaces (after Givental)}
\label{sec:givenloc}

In this appendix, we recall Givental's localization calculation
\cite{Gi01a} (see also \cite{LePa04P} for a more leisurely treatment),
which proves that the CohFT from equivariant $\PP^m$ can be obtained
from the trivial theory via a specific $R$-matrix action.
We first recall localization in the space of stable maps to $\PP^m$,
in Section~\ref{sec:locdefs}.
Next, in Section~\ref{sec:locgen}, we group the localization
contributions according to the dual graph of the source curve.
We collect identities following from the string and dilaton equation
in Section~\ref{sec:SD}, before applying them to finish the
computation in Section~\ref{sec:loctofrob}.
In Section~\ref{sec:Rchar}, we collect properties of the $R$-matrix.
Finally, in Section~\ref{sec:mirror}, we give Givental's mirror
description \cite{Gi01b} of the $R$-matrix.

\subsection{Localization in the space of stable maps}
\label{sec:locdefs}

Let $T = (\CC^*)^{m + 1}$ act diagonally on $\PP^m$.
The equivariant Chow rings of a point $\mathrm{pt}$ and $\PP^m$ are
given by
\begin{align*}
  A^*_{T} (\mathrm{pt})
  \cong& \QQ[\lambda_0, \dotsc, \lambda_m], \\
  A^*_{T} (\PP^m)
  \cong& \QQ[H, \lambda_0, \dotsc, \lambda_m]/\prod_{i = 0}^m (H - \lambda_i),
\end{align*}
where $H$ is a lift of the hyperplane class.
Furthermore, let $\eta$ be the equivariant Poincaré pairing.

There are $m + 1$ fixed points $p_0, \dotsc, p_m$ for the $T$-action
on $\PP^m$.
The characters of the action of $T$ on the tangent space
$T_{p_i}\PP^m$ are given by $\lambda_i - \lambda_j$ for $j \not= i$.
Hence the corresponding equivariant Euler class $e_i$ is
\begin{equation*}
  e_i = \prod_{j \neq i} (\lambda_i - \lambda_j).
\end{equation*}
The equivariant class $e_i$ also serves as the inverse of the norms
of the equivariant (classical) idempotents
\begin{equation*}
  \phi_i = e_i^{-1} \prod_{j \neq i} (H - \lambda_j).
\end{equation*}

The virtual localization formula \cite{GrPa99} implies that the
virtual fundamental class can be split into a sum
\begin{align*}
  [\Mbar_{g, n}(\PP^m, d)]^\vir_T
  = \sum_X \iota_{X, *} \frac{[X]^\vir_T}{e_T(N^\vir_{X, T})}
\end{align*}
of contributions of fixed loci $X$.
Here, $N^\vir_{X, T}$ denotes the virtual normal bundle of $X$ in
$\Mbar_{g, n}(\PP^m, d)$ and $e_T$ the equivariant Euler class.
Because of the denominator, the fixed-point contributions are defined
only after localizing by the elements $\lambda_0, \dotsc, \lambda_m$.
By studying the $\CC^*$-action on deformations and obstructions of
stable maps, $e_T(N^\vir_{X, T})$ can be computed explicitly.

The fixed loci can be labeled by certain decorated graphs.
These consist of
\begin{itemize}
\item a graph $(V, E)$,
\item an assignment $\zeta\colon V \to \{p_0, \dotsc, p_m\}$ of fixed points,
\item a genus assignment $g\colon V \to \ZZ_{\ge 0}$,
\item a degree assignment $d\colon E \to \ZZ_{> 0}$,
\item an assignment $p\colon \{1, \dotsc n\} \to V$ of marked points,
\end{itemize}
such that the graph is connected and contains no self-edges, two
adjacent vertices are not assigned to the same fixed point and we have
\begin{equation*}
  g = h^1(\Gamma) + \sum_{v \in V} g(v), \qquad
  d = \sum_{e \in E} d(e).
\end{equation*}
A vertex $v \in V$ is called stable if $2g(v) - 2 + n(v) > 0$, where
$n(v)$ is the number of outgoing edges at $v$.

The fixed locus corresponding to a graph is characterized by the
conditions that stable vertices $v \in V$ of the graph correspond to
contracted genus $g(v)$ components of the domain curve and that edges
$e \in E$ correspond to multiple covers of degree $d(e)$ of the torus
fixed line between two fixed points.
An unstable vertex $v$ with $(g(v), n(v)) = (0, 2)$ corresponds to
either a node connecting two multiple covers of a line or a marking at
the end of one of the multiple covers.
When $(g(v), n(v)) = (0, 2)$, the vertex $v$ corresponds to an unmarked
point at the end of a multiple cover.
Such a fixed locus is isomorphic to a product of moduli spaces of
curves
\begin{equation*}
  \prod_{v \in V} \Mbar_{g(v), n(v)}
\end{equation*}
up to a finite map.

For a fixed locus $X$ corresponding to a given graph, the Euler class
$e_T(N^\vir_{X, T})$ is a product of factors corresponding to the
geometry of the graph:
\begin{multline}
  \label{eq:normal}
  e_T(N^\vir_{X, T}) = \prod_{v\text{, stable}} \frac{e(\mathbb E^*
    \otimes T_{\PP^m, \zeta(v)})}{e_{\zeta(v)}}
  \prod_{\text{nodes}} \frac{e_{\zeta}}{-\psi_1 - \psi_2} \\
  \prod_{\substack{g(v) = 0 \\ n(v) = 1}} (-\psi_v) \prod_e
  \mathrm{Contr}_e
\end{multline}
In the first product, $\mathbb E^*$ denotes the dual of the Hodge
bundle, $T_{\PP^m, \zeta(v)}$ is the tangent space of $\PP^m$ at
$\zeta(v)$, and all bundles and Euler classes should be considered
equivariantly.
The second product is over nodes forced onto the domain curve by the
graph.
They correspond to stable vertices together with an outgoing edge, or
vertices $v$ of genus $0$ with $n(v) = 2$.
By $\psi_1$ and $\psi_2$, we denote the (equivariant) cotangent line
classes at the two sides of the node.
For example, the equivariant cotangent line class $\psi$ at a fixed
point $p_i$ on a line mapped with degree $d$ to a fixed line is more
explicitly given by
\begin{equation*}
  -\psi = \frac{\lambda_i - \lambda_j}d,
\end{equation*}
where $p_j$ is the other fixed point on the fixed line.
The explicit expressions for the last two factors of \eqref{eq:normal}
can be found in \cite{GrPa99}, but will play no role for us.
It is only important that they depend only on local data.

\subsection{General procedure}
\label{sec:locgen}

We set $W$ to be $A^*_{T} (\PP^m)$ with all equivariant parameters
localized.
For $v_1, \dotsc, v_n \in W$ the (full) CohFT $\Omega_{g, n}$ from
equivariant $\PP^m$ is defined by
\begin{multline}
  \label{eq:PmfullCohFT}
  \Omega_{g, n}^t(v_1, \dotsc, v_n) \\
  = \sum_{d, k = 0}^\infty \frac{q^d}{k!}
  \pi_* p_*\left(\prod_{i = 1}^n \ev_i^*(v_i) \prod_{i = n + 1}^{n + k}
    \ev_i^*(t) \cap [\Mbar_{g, n + k}(\PP^m, d)]^\vir\right),
\end{multline}
where
\begin{equation*}
  t = t_0 \phi_0 + \dotsb + t_m \phi_m
\end{equation*}
is a formal point on $W$, the map $\pi$ forgets the last $k$ markings
and $p$ forgets the map.
We want to calculate the push-forward via virtual localization.
In the end we will arrive at the formula of the $R$-matrix action as
described in Section~\ref{sec:cohft:reconstr} for an
endomorphism-valued power series $R_{\PP^m}(z)$.
In the following we will systematically suppress the dependence on $t$
in the notation.
In the main part of this paper, we are interested in the case where
$t$ is set to zero.

For the localization computation, we start by remarking that for each
localization graph for \eqref{eq:PmfullCohFT} there exists a dual
graph of $\Mbar_{g, n}$ corresponding to the topological type of the
stabilization under $\pi p$ of a generic source curve of that locus.
What gets contracted under the stabilization maps are trees of
rational curves.
There are three types of these unstable trees of rational curves:
\begin{enumerate}
\item \label{tree1} Those which contain one of the $n$ markings and
  are connected to a stable component
\item \label{tree2} Those which are connected to two stable components
  and contain none of the $n$ markings
\item \label{tree3} Those which are connected to one stable component
  but contain none of the $n$ markings
\end{enumerate}
Here, a stable component is a component of the source curve not
contracted by the stabilization for $\pi p$.
The first two types of trees correspond, respectively, to the preimage
under the stabilization of
\begin{enumerate}
\item one of the $n$ markings,
\item one of the nodes.
\end{enumerate}

Each type of tree gives rise to a series of localization
contributions, and we want to record it using the fact that the same
contributions already occur in genus zero.

Let $W'$ be an abstract free module over the same base ring as $W$
with a basis $w_0, \dotsc, w_m$ labeled by the fixed points of the
$T$-action on $\PP^m$.
We will later identify $W'$ with $W$ in a nontrivial way (see
Equation~\eqref{eq:WW}).
The type (\ref{tree1}) contributions are recorded by
\begin{equation*}
  \widetilde R^{-1} = \sum_i \widetilde R_i^{-1} w_i \in \Hom(W, W')[\![z]\!],
\end{equation*}
the homomorphism-valued power series such that
\begin{equation*}
  \widetilde R_i^{-1}(v) = \eta(e_i\phi_i, v) + \sum_{d, k = 0}^\infty \frac{q^d}{k!} \sum_{\Gamma \in G_{d, k, i}^1} \frac 1{\Aut(\Gamma)} \Contr_\Gamma(v)
\end{equation*}
where $G_{d, k, i}^1$ is the set of localization graphs for
$\Mbar_{0, 2 + k}(\PP^m, d)$ such that the first marking is at a
valence 2 vertex at fixed point $i$ and $\Contr_\Gamma(v)$ is the
contribution for graph $\Gamma$ for the integral
\begin{equation*}
  \int\limits_{\Mbar_{0, 2 + k}(\PP^m, d)} \frac{e_i}{-z - \psi_1} \ev_2^*(v) \prod_{l = 3}^{2 + k} \ev_l^*(t).
\end{equation*}
We define the integral in the case $(d, k) = (0, 0)$ to be zero and
will do likewise for other integrals over nonexisting moduli spaces.

The type (\ref{tree2}) contributions are recorded by the bivector
\begin{equation*}
  \widetilde V = \sum_i \widetilde V^{ij} w_i \otimes w_j \in W'^{\otimes 2} [\![z, w]\!]
\end{equation*}
which is defined by
\begin{equation*}
  \widetilde V^{ij} = \sum_{d, k = 0}^\infty \frac{q^d}{k!} \sum_{\Gamma \in G_{d, k, i, j}^2} \frac 1{\Aut(\Gamma)} \Contr_\Gamma,
\end{equation*}
where $G_{d, k, i, j}^2$ is the set of localization graphs for
$\Mbar_{0, 2 + k}(\PP^m, d)$ such that the first and second marking
are at valence 2 vertices at fixed points $i$ and $j$, respectively,
and $\Contr_\Gamma$ is the contribution for graph $\Gamma$ for the
integral
\begin{equation*}
  \int\limits_{\Mbar_{0, 2 + k}(\PP^m, d)} \frac{e_ie_j}{(-z - \psi_1)(-w - \psi_2)} \prod_{l = 3}^{2 + k} \ev_l^*(t).
\end{equation*}
Finally, the type (\ref{tree3}) contribution is a vector
\begin{equation*}
  \widetilde T = \sum_i \widetilde T_i w_i \in W'[\![z]\!]
\end{equation*}
which is defined by
\begin{equation*}
  \widetilde T_i = t + \sum_{d, k = 0}^\infty \frac{q^d}{k!} \sum_{\Gamma \in G_{d, k, i}^3} \frac 1{\Aut(\Gamma)} \Contr_\Gamma
\end{equation*}
where $G_{d, k, i}^3$ is the set of localization graphs for
$\Mbar_{0, 1 + k}(\PP^m, d)$ such that the first marking is at a
valence 2 vertex at fixed point $i$ and $\Contr_\Gamma(v)$ is the
contribution for graph $\Gamma$ for the integral
\begin{equation*}
  \int\limits_{\Mbar_{0, 1 + k}(\PP^m, d)} \frac{e_i}{-z - \psi} \prod_{l = 2}^{1 + k} \ev_l^*(t).
\end{equation*}

With these contributions we can write the CohFT already in a form
quite similar to the $R$-matrix action formula \eqref{eq:reconstr}.
Let $\omega_{g, n}$ be the $n$-form on $W'$ which vanishes if $w_i$
and $w_j$ for $i \neq j$ are inputs, which satisfies
\begin{equation*}
  \omega_{g, n}(w_i, \dotsc, w_i) = \frac{e(\mathbb E^* \otimes T_{\PP^m, p_i})}{e_i} = e_i^{g - 1} \prod_{j \neq i} c_{\lambda_j - \lambda_i}(\mathbb E)
\end{equation*}
and which for $n = 0$ is defined similarly as in
Example~\ref{ex:trivial}.
We have
\begin{equation}
  \label{eq:locresA}
  \Omega_{g, n}^t(v_1, \dotsc, v_n) = \sum_{\Gamma} \frac 1{\Aut(\Gamma)} \xi_*\left(\prod_v \sum_{k = 0}^\infty \frac 1{k!} \pi_* \omega_{g_v, n_v + k}(\dots)\right),
\end{equation}
where we put
\begin{enumerate}
\item $\widetilde R^{-1}(\psi)(v_i)$ into the argument corresponding to marking $i$,
\item one half of $\widetilde V(\psi_1, \psi_2)$ into an argument
  corresponding to a node and
\item $\widetilde T(\psi)$ into all additional arguments.
\end{enumerate}

We will still need to apply the string and dilaton equation in order
to make $\widetilde T(\psi)$ a multiple of $\psi^2$, like the
corresponding series in the reconstruction, express the Hodge classes
via Mumford's formula and then relate the series to the $R$-matrix.

\subsection{String and Dilaton Equation}
\label{sec:SD}

We want to use the string and dilaton equation to bring a series
\begin{equation}
  \label{eq:SDstart}
  \sum_{k = 0}^\infty \frac 1{k!}
  \pi_*\left( \prod_{i = 1}^n \frac 1{-x_i - \psi_i} \prod_{i =
      n + 1}^{n + k}
    Q(\psi_i)\right),
\end{equation}
where $\pi\colon \Mbar_{g, n + k} \to \Mbar_{g, n}$ is the forgetful
map and $Q = Q_0 + zQ_1 + z^2Q_2 + \dotsb$ is a formal series, into a
canonical form.

By the string equation, \eqref{eq:SDstart} is annihilated by
\begin{equation*}
  \mathcal L' = \mathcal L + \sum_{i = 1}^n \frac 1{x_i},
\end{equation*}
where $\mathcal L$ is the string operator
\begin{equation*}
  \mathcal L
  = \frac\partial{\partial Q_0} - Q_1 \frac\partial{\partial Q_0} - Q_2 \frac\partial{\partial Q_1} - Q_3 \frac\partial{\partial Q_2} - \dotsb.
\end{equation*}
Moving along the string flow for some time $-u$, that is applying
$e^{t\mathcal L'}|_{t = -u}$ to \eqref{eq:SDstart}, gives
\begin{equation*}
  \sum_{k = 0}^\infty \frac 1{k!}
  \pi_*\left( \prod_{i = 1}^n \frac{e^{-\frac u{x_i}}}{-x_i - \psi_i} \prod_{i =
      n + 1}^{n + k}
    Q'(\psi_i)\right),
\end{equation*}
for a new formal series $Q' = Q'_0 + zQ'_1 + z^2Q'_2 + \dotsb$.
In the case that
\begin{equation*}
  u
  = \sum_{k = 1}^\infty \frac 1{k!} \int\limits_{\Mbar_{0, 2 + k}} \prod_{i = 3}^{2 + k} Q(\psi_i),
\end{equation*}
which we will assume from now on, the new series $Q'$ satisfies
$Q'_0 = 0$.
This is because the string equation implies $\mathcal Lu = 1$ and
therefore applying $e^{t\mathcal L}|_{t = -u}$ to $u$ gives, on the
one hand, $u - t|_{t = -u} = 0$ and, on the other hand, the definition
of $u$ with $Q$ replaced by $Q'$, which for dimension reasons is a
nonzero multiple of $Q'_0$.

Next, by applying the dilaton equation we can remove the linear part
from the series $Q'_0$
\begin{multline}
  \label{eq:SDstable}
  \sum_{k = 0}^\infty \frac 1{k!}
  \pi_*\left( \prod_{i = 1}^n \frac 1{-x_i - \psi_i} \prod_{i =
      n + 1}^{n + k}
    Q(\psi_i)\right) \\
  = \sum_{k = 0}^\infty \frac{\Delta^{\frac{2g - 2 + n + k}2}}{k!}
  \pi_*\left( \prod_{i = 1}^n \frac{e^{-\frac u{x_i}}}{-x_i - \psi_i} \prod_{i =
      n + 1}^{n + k}
    Q''(\psi_i)\right),
\end{multline}
where $Q'' = Q' - Q'_1z$ and
\begin{equation*}
  \Delta^{\frac 12} = (1 - Q'_1)^{-1} = \sum_{k = 0}^\infty \frac 1{k!} \int\limits_{\Mbar_{0, 3 + k}} \prod_{i = 4}^{3 + k} Q(\psi_i).
\end{equation*}

We will also need identities in the degenerate cases $(g, n) = (0, 2)$
and $(g, n) = (0, 1)$.
In the first case, there is the identity
\begin{equation}
  \label{eq:sflow2}
  \frac 1{-z - w} + \sum_{k = 1}^\infty \frac 1{k!} \int\limits_{\Mbar_{0, 2 + k}} \frac 1{-z - \psi_1} \frac 1{-w - \psi_2} \prod_{i = 3}^{2 + k} Q(\psi_i)
  = \frac{e^{-u/z + -u/w}}{-z - w}.
\end{equation}
In order to see that \eqref{eq:sflow2} is true, we use that the
left-hand side is annihilated by $\mathcal L + 1/z + 1/w$ in order to
move from $Q$ to $Q'$ via the string flow and notice that there all
the integrals vanish for dimension reasons.
Similarly, there is the identity
\begin{multline}
  \label{eq:sflowT}
  1 - \frac{Q(z)}{z} - \frac 1z \sum_{k = 2}^\infty \frac 1{k!} \int\limits_{\Mbar_{0, 1 + k}} \prod_{i = 2}^{1 + k} Q(\psi_i) \\
  = e^{-u/z} \left(1 - \frac{Q'(z)}{z}\right) = e^{-u/z} \left(\Delta^{-\frac 12} - \frac{Q''(z)}{z}\right),
\end{multline}
which can be proven like the previous identity by using that the
left-hand side is annihilated by $\mathcal L + 1/z$.

We define the functions $u_i$ and $(\Delta_i/e_i)^{1/2}$ for
$i \in \{0, \dotsc, m\}$ to be the $u$ and $\Delta^{1/2}$ at the
points $Q = \widetilde T_i$ from the previous section.

\subsection{Computation of the localization series}
\label{sec:loctofrob}

We apply \eqref{eq:SDstable} to \eqref{eq:locresA} and obtain
\begin{equation}
  \label{eq:locresB}
  \Omega_{g, n}^t(v_1, \dotsc, v_n) = \sum_{\Gamma} \frac 1{\Aut(\Gamma)} \xi_*\left(\prod_v \sum_{k = 0}^\infty \frac 1{k!} \pi_* \omega'_{g_v, n_v + k}(\dots)\right),
\end{equation}
where we put
\begin{enumerate}
\item $R^{-1}(\psi)(v_i)$ into the argument corresponding to marking $i$,
\item one half of $V(\psi_1, \psi_2)$ into an argument
  corresponding to a node and
\item $T(\psi)$ into all additional arguments.
\end{enumerate}
Here $R^{-1}$, $V$ and $T$ are defined exactly as $\widetilde R^{-1}$,
$\widetilde V$ and $\widetilde T$ but with the replacement
\begin{equation*}
  \frac{e_i}{-x - \psi} \rightsquigarrow \frac{e_i e^{-\frac{u_i}x}}{-x - \psi}
\end{equation*}
made at the factors we put at the ends of the trees.
The form $\omega'_{g, n}$ satisfies
\begin{equation*}
  \omega'_{g, n}(w_i, \dotsc, w_i) = \Delta_i^{\frac{2g - 2 + n}2} e_i^{-\frac n2} \prod_{j \neq i} c_{\lambda_j - \lambda_i}(\mathbb E).
\end{equation*}

We now want to compute $R^{-1}$, $V$ and $T$ in terms of the
homomorphism-valued series $S^{-1}(z) \in \Hom(W, W')[\![z]\!]$ with
$w_i$-component
\begin{multline*}
  S_i^{-1}(z) = \langle \frac{e_i\phi_i}{-z - \psi}, -\rangle \\
  := \eta(e_i\phi_i, -) + \sum_{d, k = 0}^\infty \frac{q^d}{k!} \int\limits_{\Mbar_{0, 2 + k}(\PP^m, d)} \frac{\ev_1^*(e_i\phi_i)}{-z - \psi_1} \ev_2^*(-) \prod_{j = 3}^{k + 2} \ev_j^*(t).
\end{multline*}

We start by computing $S^{-1}$ via localization.
Using that in genus zero, the Hodge bundle is trivial, we find that at
the vertex with the first marking we need to compute integrals exactly
as in \eqref{eq:sflow2}, where the first summand stands for the case
that the vertex is unstable and the second summand stands for the case
that the vertex is stable with $k$ trees of type \ref{tree3} and one
tree of type \ref{tree1} attached to it.
Applying \eqref{eq:sflow2}, we obtain
\begin{equation}
  \label{eq:SvR}
  S_i^{-1}(z) = e^{-\frac{u_i}z} R_i^{-1}(z).
\end{equation}

Using the shorthand notation
\begin{multline*}
  \left\langle \frac{v_1}{x_1 - \psi}, \frac{v_2}{x_2 - \psi}, \frac{v_3}{x_3 - \psi}\right\rangle \\
  := \sum_{d, k = 0}^\infty \frac{q^d}{k!}
  \int\limits_{\Mbar_{0, 3 + k}(\PP^m, d)}
  \frac{\ev_1^* v_1}{x_1 - \psi_1} \frac{\ev_2^* v_2}{x_2 - \psi_2} \frac{\ev_3^* v_3}{x_3 - \psi_3} \prod_{i = 4}^{3 + k} \ev_i^* (t)
\end{multline*}
for genus zero Gromov--Witten invariants and applying the string
equation, we can also write
\begin{equation*}
  S_i^{-1}(z) = -\frac 1z \langle \frac{e_i\phi_i}{-z - \psi}, \mathbf 1, -\rangle.
\end{equation*}
By the identity axiom and Witten-Dijgraaf-Verlinde-Verlinde equation,
we have
\begin{multline*}
  \langle \frac{e_i\phi_i}{-z - \psi}, \frac{e_j\phi_j}{-w - \psi}, \mathbf 1\rangle = \langle \frac{e_i\phi_i}{-z - \psi}, \frac{e_j\phi_j}{-w - \psi}, \bullet\rangle \langle \bullet, \mathbf 1, \mathbf 1\rangle \\
  = \langle \frac{e_i\phi_i}{-z - \psi}, \mathbf 1, \bullet\rangle \langle \bullet, \mathbf 1, \frac{e_j\phi_j}{-w - \psi}\rangle,
\end{multline*}
where in the latter two expressions the $\bullet$ should be filled with
$\eta^{-1}$, so that
\begin{multline}
  \label{eq:SetaS}
  \frac{S_i^{-1}(z) \otimes S_j^{-1}(w)}{-z - w} \eta^{-1} \\
  = \frac{\eta(e_i\phi_i, e_j\phi_j)}{-z - w} + \sum_{d, k = 0}^\infty \frac{q^d}{k!} \int\limits_{\Mbar_{0, 2 + k}(\PP^m, d)} \frac{\ev_1^*(e_i\phi_i)}{-z - \psi_1} \frac{\ev_2^*(e_j\phi_j)}{-w - \psi_2} \prod_{l = 3}^{k + 2} \ev_l^*(t).
\end{multline}
We compute the right-hand side via localization.
There are two cases in the localization depending on whether the first
and second markings are at the same vertex or at different ones.
In the first case, we apply \eqref{eq:sflow2} at this common vertex
and obtain the total contribution
\begin{equation*}
  \frac{e_i\delta_{ij} e^{-\frac{u_i}z - \frac{u_j}w}}{-z - w},
\end{equation*}
which includes the unstable summand.
In the other case, we apply \eqref{eq:sflow2} at the two vertices and
obtain
\begin{equation*}
  e^{-\frac{u_i}z - \frac{u_j}w} V^{ij}(z, w).
\end{equation*}
So all together
\begin{equation*}
  V^{ij}(z, w) = \frac{(R_i^{-1}(z) \otimes R_j^{-1}(w)) \eta^{-1} - e_i\delta_{ij}}{-z - w}.
\end{equation*}

Finally, we express $T$ in terms of $R$ by computing
\begin{equation*}
  S_i^{-1}(z)\mathbf 1 = 1 - \frac{t_i}z - \frac 1z\sum_{d, k = 0}^\infty \frac{q^d}{k!} \int\limits_{\Mbar_{0, 1 + k}(\PP^m, d)} \frac{\ev_1^*(e_i\phi_i)}{-z - \psi_1} \prod_{j = 2}^{1 + k} \ev_j^*(t)
\end{equation*}
via localization.
Applying \eqref{eq:sflowT} at the first marking, we find that
\begin{equation*}
  S_i^{-1}(z)\mathbf 1 = e^{-\frac{u_i}z} \left(\Delta_i^{-\frac 12} e_i^{\frac 12} - \frac{T_i(z)}z\right).
\end{equation*}
So
\begin{equation*}
  T(z) = z\left(\sum_i \Delta_i^{-\frac 12} e_i^{\frac 12}w_i - R^{-1}(z) \mathbf 1\right).
\end{equation*}

By \eqref{eq:locresB}, the underlying TQFT of $\Omega^t_{g, 0}$ is
given by
\begin{equation*}
  \sum_i \Delta_i^{g - 1}.
\end{equation*}
This implies that the $\Delta_i$ need to be the inverses of the norms
of the idempotents for the quantum product of equivariant $\PP^m$
(because these are pairwise different).
Since $\widetilde T$ vanishes at $(t, q) = 0$, we have that $\Delta_i$
agrees with $e_i$ at $(t, q) = 0$.
Therefore, we can identify $W'$ with $W$ via
\begin{equation}
  \label{eq:WW}
  W' \to W, \qquad w_i \mapsto \sqrt{\Delta_i/e_i}\ \epsilon_i,
\end{equation}
where $\epsilon_i$ is the idempotent element which coincides with
$\phi_i$ at $(t, q) = 0$.
The previous results then say exactly that $\Omega^t$ is obtained from
the CohFT $\omega'$ by the action of the $R$-matrix $R$.
In turn, Example~\ref{ex:mumford} implies that $\omega'$ is obtained
from the TQFT by the action of an $R$-matrix $R_{\text{Mumford}}$
which is diagonal in the basis of idempotents and has entries
\begin{equation*}
  \exp\left(\sum_{i = 1}^\infty \frac{B_{2i}}{2i(2i - 1)} \sum_{j \neq i} \left(\frac z{\lambda_j - \lambda_i}\right)^{2i - 1}\right).
\end{equation*}
We define the endomorphism $R_{\PP^m}$ to be the composition
\begin{equation*}
  R_{\PP^m} = R \cdot R_{\text{Mumford}}.
\end{equation*}
Therefore, the $R$-matrix action of $R_{\PP^m}$ takes the TQFT to
$\Omega^t$.

\subsection{Characterization of $R$-matrix}
\label{sec:Rchar}

We study the $R$-matrix $R_{\PP^m}$ defined in the previous section as
the composition of $R$ and $R_{\text{Mumford}}$.

As we have seen, there is a power series $S^{-1}(z)$ in $z^{-1}$
strongly related to $R$.
By considering \eqref{eq:SetaS} as $w + z \to 0$, we see that
$S^{-1}(z)$ satisfies the symplectic condition; that is, its inverse
$S(z)$ is the adjoint of $S^{-1}(-z)$ with respect to $\eta$.
Using this, we can say, more explicitly, that the evaluation of $S(z)$
at the $i$th normalized idempotent is the vector
\begin{equation*}
  \left\langle \frac{\sqrt{e_i}\phi_i}{z - \psi}, \eta^{-1}\right\rangle.
\end{equation*}
In terms of $S(z)$, the relation between $R$ and $S$, when written in
the basis of normalized idempotents, is
\begin{equation}
  \label{eq:SvR2}
  R(z) = S(z) e^{-\diag(u_0, \dotsc, u_m)/z}.
\end{equation}

The series $S(z)$ satisfies the \emph{quantum differential equation}
\begin{equation}
  \label{eq:QDE}
  z \frac\partial{\partial t_\mu} S(z) = \phi_\mu \star S(z)
\end{equation}
for any $\mu$.
This follows from the genus zero topological recursion relations
\begin{equation*}
  z\left\langle \phi_\mu, \frac{\sqrt{e_i}\phi_i}{z - \psi}, \eta^{-1}\right\rangle = \langle \phi_\mu, \eta^{-1}, \bullet\rangle \left\langle \bullet, \frac{\sqrt{e_i}\phi_i}{z - \psi}\right\rangle,
\end{equation*}
where, as previously, $\eta^{-1}$ should be inserted at the two
$\bullet$.

The divisor axiom in Gromov--Witten theory determines the
$q$-dependence of $S$.
We have
\begin{multline*}
  \left\langle \frac{\sqrt{e_i}\phi_i}{z - \psi}, \eta^{-1}, H\right\rangle
  = \left\langle \frac{\sqrt{e_i}H \phi_i}{z(z - \psi)}, \eta^{-1}\right\rangle + d \left\langle \frac{\sqrt{e_i}\phi_i}{z - \psi}, \eta^{-1}\right\rangle \\
  = \left(\frac{\lambda_i}z + d\right)\left\langle \frac{\sqrt{e_i}\phi_i}{z - \psi}, \eta^{-1}\right\rangle
\end{multline*}
and hence
\begin{equation}
  \label{eq:qdep}
  H \star S_i(z) = DS_i(z) + \lambda_i S_i(z),
\end{equation}
where
\begin{equation*}
  D = zq\frac\partial{\partial q}.
\end{equation*}

The analog of equation~\eqref{eq:SvR2} defines a new series
$S_{\PP^1}$.
It satisfies the same symplectic condition, quantum differential
equation and $q$-dependence \eqref{eq:qdep} since $R_{\text{Mumford}}$
is diagonal and does not depend on the $t_i$.

Let us consider the classical limit $(t, q) \to 0$.
In this case, $\Delta_i = e_i$, $w_i = \phi_i$ and
$\widetilde R = R = \Id$.
Therefore,
\begin{equation}
  \label{eq:limit}
  R_{\PP^1}|_{(t, q) = 0} = R_{\text{Mumford}} = \exp(\diag(b_0, \dotsc, b_m)),
\end{equation}
where
\begin{equation*}
  b_j = \sum_{i = 1}^\infty \frac{B_{2i}}{2i(2i - 1)} \sum_{l \neq j} \left(\frac z{\lambda_l - \lambda_j}\right)^{2i - 1}.
\end{equation*}

\begin{proposition}
  The series $R_{\PP^1}(z)$ written in the basis of normalized
  idempotents is the unique matrix-valued power series in $z$
  depending on parameters $\lambda_0, \dotsc, \lambda_m$, $q$ and
  $t_0, \dotsc, t_m$ such that there exist functions $u_0, \dotsc,
  u_m$ independent of $z$ and a matrix-valued power series $S_{\PP^1}$
  in $z^{-1}$ such that
  \begin{enumerate}
  \item the series $R_{\PP^1}$ and $S_{\PP^1}$ are related by
    \eqref{eq:SvR2},
  \item $S_{\PP^1}$ satisfies the quantum differential equation
    \eqref{eq:QDE},
  \item $S_{\PP^1}$ satisfies \eqref{eq:qdep} and
  \item the classical limit $(t, q) \to 0$ of $R_{\PP^1}$ is given by
    \eqref{eq:limit}.
  \end{enumerate}
\end{proposition}
\begin{proof}
  In this Section we have checked that $R_{\PP^1}$ indeed satisfies
  these properties.

  By a general analysis \cite[Proposition~6.7]{Gi01b} of the quantum
  differential equation we see that the first two properties determine
  $R_{\PP^1}$ up to functions constant in the $t_i$.
  Via \eqref{eq:SvR2}, equation \eqref{eq:qdep} also determines the
  dependence of $R_{\PP^1}$ on $q$.
  Hence $R_{\PP^1}$ is uniquely determined from its classical limit.
\end{proof}

\subsection{Small equivariant mirror}
\label{sec:mirror}

Following \cite{Gi01b}, we give a concrete description of the
$R$-matrix for the equivariant Gromov--Witten theory of $\PP^m$.
In contrast to the previous sections of the appendix, we will soon
restrict ourselves to the small quantum cohomology case $t = 0$.
To simplify the notation, we will leave out the subscript $\PP^m$ in
$R_{\PP^m}$ and $S_{\PP^m}$.

If we use alternative coordinates
\begin{equation*}
  t = \widetilde t_0 + \widetilde t_1 H + \dotsb + \widetilde t_m H^m,
\end{equation*}
the quantum differential equation implies
\begin{equation}
  \label{eq:QDE2}
  z \frac\partial{\partial\widetilde t_1} S_i(z) = H \star S_i(z) = DS_i(z) + \lambda_i S_i(z).
\end{equation}
So we can rewrite the differential equation \eqref{eq:QDE2} satisfied
by $S_i$ to
\begin{equation}
  \label{eq:QDE3}
  \left(zq\frac\partial{\partial q} + \lambda_i \right)S_i(z) = H \star S_i(z).
\end{equation}
Now it makes sense to restrict to $t = 0$ and we will do so from now
on (without changing the notation).

If we write $S_{\mu i} = \eta(S_i, \phi_\mu)$, we can rewrite
\eqref{eq:QDE3} to
\begin{equation}
  \label{eq:QDE4}
  D(S_{\mu i} e^{\ln(q) \lambda_i/z})
  = S_{(\mu + 1)i} e^{\ln(q) \lambda_i/z}
\end{equation}
for $\mu < m$ and
\begin{equation}
  \label{eq:QDE5}
  \prod_{j = 0}^m (D - \lambda_j) (S_{0i} e^{\ln(q) \lambda_i/z})
  = qS_{0i} e^{\ln(q) \lambda_i/z}.
\end{equation}
\begin{proposition}
  The series $R$ is the unique matrix-valued power series in $z$
  depending on parameters $\lambda_0, \dotsc, \lambda_m$, $q$ such
  that there exist functions $u_0, \dotsc, u_m$ independent of $z$ and
  a matrix-valued power series $S$ in $z^{-1}$ such that
  \begin{enumerate}
  \item the series $R$ and $S$ are related by \eqref{eq:SvR2},
  \item $S$ satisfies the quantum differential equations
    \eqref{eq:QDE4} and \eqref{eq:QDE5},
  \item the classical limit $q \to 0$ of $R$ is given by
    \eqref{eq:limit}.
  \end{enumerate}
\end{proposition}
\begin{proof}
  Equation \eqref{eq:QDE4} determines all of $S$ from $S_{0i}$, which
  is determined from \eqref{eq:QDE5} up to additive integration
  constants depending on $\lambda_0, \dotsc, \lambda_m$.
  Choosing the integration constants such that the classical limit
  condition holds determines $S$ up to constants (not depending on
  $\lambda_0, \dotsc, \lambda_m$), which is enough to determine $R$
  uniquely via \eqref{eq:SvR2}.
\end{proof}

Givental constructs asymptotic solutions $S_{0i}$ to \eqref{eq:QDE5}
via oscillating integrals on the mirror manifold
\begin{equation*}
  \{(T_0, \dotsc, T_m): e^{T_0}\cdot\dots\cdot e^{T_m} = q\} \subset \CC^{m + 1}
\end{equation*}
with superpotential
\begin{equation*}
  F(T) = \sum_{j = 0}^m (e^{T_j} + \lambda_jT_j).
\end{equation*}
The integrals are given by
\begin{equation*}
  S_{0i} e^{\ln(q) \lambda_i/z}
  = (-2\pi z)^{-m/2} \int\limits_{\Gamma_i \subset \{\sum T_j = \ln q\}} e^{F(T)/z} \omega
\end{equation*}
along $m$-cycles $\Gamma_i$ through a specific critical point of the
superpotential $F$ constructed using the Morse theory of the real part
of $F/z$.
The form $\omega$ is the restriction of
$\mathrm dT_0 \wedge \dotsb \wedge \mathrm dT_m$.
To see that the integrals are actual solutions, notice that applying
$D - \lambda_j$ to the integral has the same effect as multiplying the
integrand by $e^{T_j}$.

There are $m + 1$ critical points at which it is possible to do a
stationary phase expansion of $S_{0i}$.
Let us write $P_i$ for the solution to
\begin{equation*}
  \prod_{i = 0}^m (X - \lambda_i) = q
\end{equation*}
with limit $\lambda_i$ as $q \to 0$.
For each $i$, we need to choose the critical point
$e^{T_j} = P_i - \lambda_j$ in order for the factor\footnote{The
  function $u_i$ of this appendix does not completely agree with the
  function $u_i$ of Section~\ref{sec:SD}.
  They differ by the constant
  $c = \sum_{j \neq i} (\lambda_i - \lambda_j)(-1 + \ln(\lambda_i -
  \lambda_j))$.
  Correspondingly, the $S_i$ defined via the oscillating integral also
  coincides with the series from localization only up to a factor
  $e^{c/z}$.
  These differences are irrelevant for the computation of the
  $R$-matrix.}
\begin{equation*}
  \exp(u_i / z) := \exp\left(\left(\sum_{j = 0}^m (P_i - \lambda_j + \lambda_j \ln(P_i - \lambda_j)) - \lambda_i\ln(q)\right)/z\right)
\end{equation*}
of $S_{0i}$ to be well-defined in the limit as $q \to 0$.
Shifting the integral to the critical point and scaling coordinates by
$\sqrt{-z}$, we find
\begin{equation}
  \label{eq:Sasymp}
  S_{0i} = e^{u_i/z} \int \exp\left(-\sum_j (P_i - \lambda_j) \sum_{k = 3}^\infty \frac{T_j^k (-z)^{(k - 2)/2}}{k!}\right) \mathrm d\mu_i
\end{equation}
for the conditional Gaussian distribution
\begin{equation*}
  \mathrm d\mu_i = (2\pi)^{-m/2} \exp\left(-\sum_j (P_i - \lambda_j) \frac{T_j^2}2\right) \omega.
\end{equation*}
For the asymptotic expansion, we formally expand the exponential in
\eqref{eq:Sasymp} and integrate over the real part of the image of the
mirror under the transformations we have performed.
The integrals are moments of $\mu_i$ which can be computed using the
covariance matrix
\begin{equation*}
  \sigma_i(T_k, T_l)
  = \frac 1{\Delta_i}
  \begin{cases}
    -\prod_{j \notin \{k, l\}} (P_i - \lambda_j), & \text{for }k \neq l, \\
    \sum_{m \neq k} \prod_{j \notin \{k, m\}} (P_i - \lambda_j), & \text{for }k = l.
  \end{cases}
\end{equation*}
Since odd moments of Gaussian distributions, vanish we find that the
asymptotic expansion of $e^{-u_i/z} S_{0i}$ is a power series in $z$.
So, using \eqref{eq:SvR} as a definition, by \eqref{eq:QDE4} the
entries of the $R$-matrix in the basis of normalized idempotents are
given by similar asymptotic expansion of
\begin{equation*}
  \Delta_k^{-1/2} \prod_{j \neq k}(D + \lambda_i - P_j) \ (e^{-u_i/z}S_{0i}).
\end{equation*}

We need to check that the $R$-matrix given in terms of asymptotics of
oscillating integrals behaves correctly in the limit as $q \to 0$.
By definition, in this limit, $P_i \to \lambda_i$.
By symmetry, it is enough to consider the zeroth column.
Set $x_i = e^{T_i}$.
Then
\begin{multline*}
  \lim\limits_{q \to 0} R_{j0}
  \asymp \lim\limits_{q \to 0} e^{-u_0/z} \Delta_0^{-1/2} \prod_{k \neq j}(zq\frac{\mathrm d}{\mathrm dq} + \lambda_j - \lambda_k) \ S_{00} \\
  = \lim\limits_{q \to 0} \frac{e^{-u_0/z}}{\sqrt{\Delta_0} (-2\pi z)^{m/2}} \int e^{(\sum_k (e^{T_k} - (\lambda_0 - \lambda_k) T_k)) / z + \sum_{k \neq j} T_k} \omega \\
  = \lim\limits_{q \to 0} \frac{e^{-u_0/z}}{\sqrt{\Delta_0}(-2\pi z)^{m/2}} \int e^{(\sum\limits_{k \neq 0} (x_k - (\lambda_0 - \lambda_k) T_k) + \frac q{\prod\limits_{k \neq 0} x_k}) / z} \prod_{k \neq j} x_j \bigwedge_{k = 1}^m \mathrm dT_k.
\end{multline*}
In the last step, we moved to the chart
\begin{equation*}
  x_0 = \frac q{\prod_{j \neq 0} x_j}.
\end{equation*}
Since in this chart $\lim_{q \to 0} x_0 = 0$, we have that $R_{j0}$
vanishes unless $j = 0$.
On the other hand, in the limit as $q \to 0$, the integral for
$R_{00}$ splits into one-dimensional integrals
\begin{equation*}
  \lim\limits_{q \to 0} R_{00}
  \asymp \lim\limits_{q \to 0} \frac{e^{-u_0/z}}{\sqrt{\Delta_0}(-2\pi z)^{m/2}} \prod_{k \neq 0} \int\limits_0^\infty e^{(x - (\lambda_0 - \lambda_k) \ln(x)) / z} \mathrm dx.
\end{equation*}
Let us temporarily set $z_k = -z/(\lambda_0 - \lambda_k)$.
The prefactors also split into pieces in the limit and we calculate
the factor corresponding to $k$ to be
\begin{multline*}
  \hspace{-1em}\frac{e^{(1 - \ln(\lambda_0 - \lambda_k)) / z_k}}{\sqrt{-2\pi z(\lambda_0 - \lambda_k)}} \int\limits_0^\infty e^{(x - (\lambda_0 - \lambda_k) \ln(x))/z} \mathrm dx
  = \frac{e^{(1 - \ln(1 / z_k)) / z_k}}{\sqrt{2\pi / z_k }} \Gamma\left(1 + \frac 1{z_k}\right)\\
  = \frac{e^{(1 - \ln(1 / z_k)) / z_k}}{\sqrt{2\pi z_k}} \Gamma\left(\frac 1{z_k}\right)
  \asymp \exp\left(\sum_{l = 1}^\infty \frac{B_{2l}}{2l(2l - 1)} \left(\frac z{\lambda_k - \lambda_0}\right)^{2l - 1} \right),
\end{multline*}
using Stirling's approximation of the gamma function in the last step.
So the product of the factors gives the expected limit
\eqref{eq:limit} of $R_{00}$ for $q \to 0$.

\bibliographystyle{utphys}
\bibliography{pixrel}
\addcontentsline{toc}{section}{References}

\vspace{+8 pt}
\noindent
Departement Mathematik \\
ETH Zürich \\
felix.janda@math.ethz.ch

\vspace{+8 pt}
\noindent
Institut de Mathématiques de Jussieu \\
Paris \\
felix.janda@imj-prg.fr

\end{document}